\documentclass[12pt]{amsart}
\usepackage{geometry}                
\usepackage{graphicx}
\usepackage{amssymb}
\usepackage{mathrsfs}
\usepackage{epstopdf}
\usepackage{amsmath}
\usepackage{amsthm}
\usepackage{amssymb}
\usepackage{indentfirst}

\usepackage{pgf,tikz}
\usepackage{hyperref}
\usepackage[utf8]{inputenc}
\usepackage{pgfplots}
\usepackage{mathrsfs}
\usepackage{mdframed}
\usetikzlibrary{arrows, automata, mindmap}
\usetikzlibrary{calc}
\newtheorem{theorem}{Theorem}[section]

\newtheorem*{main}{Main Theorem}
\newtheorem{lemma}[theorem]{Lemma}

\theoremstyle{definition}
\newtheorem{definition}[theorem]{Definition}
\newtheorem{remark}{Remark}

\title{Positive metric entropy arises in some nondegenerate nearly integrable systems}
\author{Dong Chen}
\address{Dong Chen: Department of Mathematics, Pennsylvania State University, University Park, PA 16802, USA}
\email{dxc360@psu.edu}
\date{}                                         

\begin{document}
\maketitle

\begin{abstract}
The celebrated KAM Theory says that if one makes a small perturbation of a non-degenerate completely integrable system, we still see a huge measure of invariant tori with quasi-periodic dynamics in the perturbed system. These invariant tori are known as KAM tori. What happens outside KAM tori draws a lot of attention. In this paper we present a Lagrangian perturbation of the geodesic flow on a flat 3-torus. The perturbation is $C^\infty$ small but the flow has a positive measure of trajectories with positive Lyapunov exponent. The measure of this set is of course extremely small. Still, the flow has positive metric entropy. From this result we get positive metric entropy outside some KAM tori.
\end{abstract}

\section{Introduction}
\footnote{2010 \emph{Mathematics Subject Classification.} 37A35, 37J40, 53C60.

\emph{Key words and phrases.} KAM theory, Finsler metric, duel lens map, Hamiltonian flow, perturbation, metric entropy.

The author was partially supported by Dmitri Burago's Department research fund 42844-1001.}

Already in the early 50's the study of nearly integrable Hamiltonian systems has drawn the attention of many outstanding mathematicians such as Arnol'd, Kolmogorov and Moser. Indeed, for any integrable Hamiltonian system the whole phase space is foliated by invariant Lagrangian submanifolds that are diffeomorphic to tori, generally called \textit{KAM tori}, and on which the dynamics is conjugated to a rigid rotation. Therefore, it is natural to ask what happens to such a foliation and to these stable motions once the system is slightly perturbed. In 1954 Kolmogorov \cite{K} - and later Arnol'd \cite{A1} and Moser \cite{M} in different contexts - proved that, for small perturbations of an integrable system it is still possible to find a big measure set of KAM tori. This result, commonly referred to as \textit{KAM theorem}, contributed to raise new interesting questions, for instance about the destiny of the stable motions that are destroyed by effect of the perturbation (in other words, about the dynamics outside KAM tori). In this context, Arnol'd \cite{A2} constructed an example of a perturbed integrable system, in which some orbits outside KAM tori have a wide range in action variables (even though the rate of change of action variables is exponentially small \cite{N}). This striking phenomenon, known as \textit{Arnol'd diffusion} and still quite far from being fully understood, shows the presence of some randomness in the dynamics outside KAM tori. The question we address in the present paper is therefore the following: how much random can the motion outside KAM tori be?"

It is well-known that, $C^2$-generically the Hamiltonian flow has positive topological entropy (cf. \cite{NH1}, see also  \cite{C} for an analogous statement for Riemannian geodesic flows). Once we turn our attention to metric entropy, the problem becomes more challenging and one cannot simply derive positive metric entropy from positive topological entropy. In fact, Bolsinov and Taimanov \cite{BT} built an example of a Riemannian manifold on which the geodesic flow has positive topological entropy but zero metric entropy. 

Recently Burago and Ivanov \cite{BI} used dual lens map to construct a reversible Finsler metric $C^{\infty}$-close to the   standard metric on $S^n, n\geq 4$, such that its geodesic flow has positive metric entropy. However the geodesic flow on the sphere is degenerate, hence it does not lie in the realm of KAM theory. 

Unlike the case of spheres, the geodesic flow on flat tori are nondegenrate. In this paper we therefore provide an example analogous to Burago-Ivanov's one on $\mathbb{T}^3$. More precisely, we prove the following:

\begin{main}
For every $\epsilon>0$ there exists a reversible Finsler metric on $\mathbb{T}^3$ which is $\epsilon$-close to the Euclidean metric in the $C^{\infty}$-sense and such that the associated geodesic flow has positive metric entropy.
\end{main}

Our theorem shows that in the complement of KAM tori, the behavior of nearly integrable Hamiltonian flows can be quite stochastic. However our example does not possess Arnold diffusion. For details, see Remark \ref{rem4}.

In order to prove the main theorem we start with perturbing the return map associated with the standard geodesic flow on a specific section to get positive metric entropy. To prove positiveness of metric entropy we use Maupertuis principle and Donnay-Burns-Gerber cap \cite{BG}\cite{D} to perturb the kinetic Hamiltonian. This method was also used by Donnay and Liverani \cite{DL}. By pulling back via the standard projection, we can perturb the return map to get positive metric entropy. Now, using  Burago-Ivanov dual lens map theory \cite{BI}, from the perturbed return map we get a (reversible) Finsler metric on $\mathbb{T}^3$ satisfying all the requirements of the main theorem. We shall however notice that, by upper semicontinuity (see \cite{NH2}), the metric entropy we get is microscopic.

\section{Preliminaries}
Let $M$ be a smooth $n$-dimensional manifold, $T^*M$ its cotangent bundle, and $\omega$ the standard symplectic form on $T^*M$. To the pair $(H,\omega)$ we can associate a unique vector field $X_H$ by
$\omega(X_H, V)=dH(V)$ for any smooth vector field $V$ on $T^*M$, which is called the \textit{Hamiltonian vector field}. The flow $\Phi_H^t$ on $T^*M$ defined by $X_H$ is called the \textit{Hamiltonian flow of $H$}. 

A typical example of a Hamiltonian flow is the geodesic flow on a Finsler manifold. Let $\varphi$ be a Finsler metric on $M$, i.e.  a smooth family of quadratically convex norms $\varphi(x, \cdot)$ on each tangent space $T_x M$. It is \textit{reversible} if $\varphi(x,v)=\varphi(x,-v)$ for all $x\in M, v\in T_x M$. Denote with $UTM$ its unit tangent bundle; the Finsler metric $\varphi$ defines a dual norm on the cotangent bundle $T^*M$ by
$$\varphi^*(\chi):=\sup_{v\in UT_xM}\{\chi(v)\}, \text{ for }\chi\in T^*_x M.$$
The geodesic flow $g_t$ on $(M,\varphi)$ is defined to be the Hamiltonian flow on $T^*M$ with Hamiltonian $(\varphi^*)^2/2$. Recall that the geodesic flow can also be viewed as the Euler-Lagrange flow on $TM$ associated with the 2-homogeneous Lagrangian $\varphi^2/2$.

One can easily see that $\Phi_H^t$ is a symplectomorphism (i.e. preserves $\omega$) and hence volume-preserving. Once we fix a level set $H^{-1}(c)$, we can define a conditional measure on this level set from the volume form. Such conditional measure is invariant under $\Phi_H^t$ and it is called the \textit{Liouville measure}.

For any point $x$ in $(M,\varphi)$, the unit ball $B_x$ in $T_x M$ is a convex body. By F. John \cite{J}, among all ellipsoids contained in $B_x$, there exists a unique ellipsoid $E_x$ with maximum volume. $E_x$ is the unit sphere of some  quadratic form on $T_x M$. In this way we can define quadratic forms on each tangent spaces and these forms are close to Finsler norms. In this way we can associate with the Finsler metric $\varphi$ a Riemannian metric $g_{\varphi}$, from which $UTM$ inherits a Riemannian structure (see \cite{S} for details). This metric is called the Sasaki metric. For each vector $\zeta\in T_{v}UTM$ we define the Lyapunov exponent by
$$\chi^{+}(v,\zeta):=\limsup_{t\rightarrow\infty} \frac{\ln||Dg_t \zeta||}{t} $$
and the upper Lyapunov exponent by
$$\chi^{+}(v):=\max_{\zeta\in T_{v}UTM}\chi^{+}(v,\zeta).$$
For our purpose, there is no need to recall the precise definition of the metric entropy $h_{\mu}$ for the Liouville measure $\mu$ on $UTM$. Indeed, it is enough to know that  Pesin's inequality \cite{P}
$$h_{\mu}\geq\int_{UTM} \chi^{+}(v)d\mu(v) \eqno(1)$$
provides a lower bound of metric entropy. Indeed, this formula tells us that the metric entropy is no less than the mean of upper Lyapunov exponent.

\begin{center}
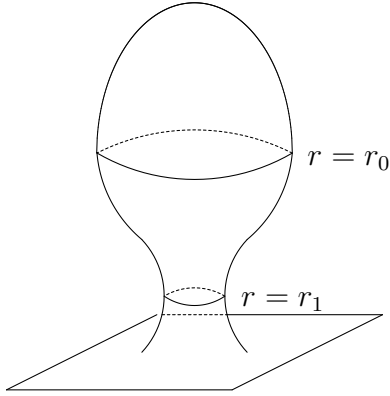
\begin{figure}[h]
\begin{tikzpicture}[line cap=round,line join=round,>=triangle 45,x=1.0cm,y=1.0cm]
\clip(-0.3,-0.1) rectangle (6.394087306319387,5.253268867919481);
\draw (0.,0.)-- (2.,1.);
\draw (5.,1.)-- (3.,0.);
\draw (3.,0.)-- (0.,0.);
\draw [shift={(1.,1.25)}] plot[domain=-0.7531512809621947:0.7531512809621944,variable=\t]({1.*1.0965856099730655*cos(\t r)+0.*1.0965856099730655*sin(\t r)},{0.*1.0965856099730655*cos(\t r)+1.*1.0965856099730655*sin(\t r)});
\draw [shift={(4.,1.25)}] plot[domain=2.388441372627599:3.8947439345519874,variable=\t]({1.*1.0965856099730653*cos(\t r)+0.*1.0965856099730653*sin(\t r)},{0.*1.0965856099730653*cos(\t r)+1.*1.0965856099730653*sin(\t r)});
\draw [shift={(2.886947292709711,3.2938768202463033)}] plot[domain=3.2259803818347432:4.0136880286584775,variable=\t]({1.*1.6898436741603213*cos(\t r)+0.*1.6898436741603213*sin(\t r)},{0.*1.6898436741603213*cos(\t r)+1.*1.6898436741603213*sin(\t r)});
\draw [shift={(2.1104217339243587,3.2999973562133778)}] plot[domain=5.409956594422004:6.194639459546497,variable=\t]({1.*1.6962234298777308*cos(\t r)+0.*1.6962234298777308*sin(\t r)},{0.*1.6962234298777308*cos(\t r)+1.*1.6962234298777308*sin(\t r)});
\draw (1.2000048072783327,3.155439043719245) -- (1.2000048072783327,3.155439043719245);
\draw (1.2000048072783327,3.155439043719245) -- (1.2065047858501403,3.349823182024358);
\draw (1.2065047858501403,3.349823182024358) -- (1.213004764421948,3.4321861553491866);
\draw (1.213004764421948,3.4321861553491866) -- (1.2195047429937556,3.495150329256437);
\draw (1.2195047429937556,3.495150329256437) -- (1.2260047215655632,3.5480307458219205);
\draw (1.2260047215655632,3.5480307458219205) -- (1.2325047001373708,3.594441448080551);
\draw (1.2325047001373708,3.594441448080551) -- (1.2390046787091784,3.6362385507577293);
\draw (1.2390046787091784,3.6362385507577293) -- (1.245504657280986,3.6745261261678692);
\draw (1.245504657280986,3.6745261261678692) -- (1.2520046358527936,3.7100244522706323);
\draw (1.2520046358527936,3.7100244522706323) -- (1.2585046144246013,3.7432344549737717);
\draw (1.2585046144246013,3.7432344549737717) -- (1.2650045929964089,3.7745212977262588);
\draw (1.2650045929964089,3.7745212977262588) -- (1.2715045715682165,3.8041609781835115);
\draw (1.2715045715682165,3.8041609781835115) -- (1.2780045501400241,3.8323681709758906);
\draw (1.2780045501400241,3.8323681709758906) -- (1.2845045287118317,3.8593137933683597);
\draw (1.2845045287118317,3.8593137933683597) -- (1.2910045072836394,3.885136582876201);
\draw (1.2910045072836394,3.885136582876201) -- (1.297504485855447,3.9099510081721314);
\draw (1.297504485855447,3.9099510081721314) -- (1.3040044644272546,3.933852840448662);
\draw (1.3040044644272546,3.933852840448662) -- (1.3105044429990622,3.9569231794016635);
\draw (1.3105044429990622,3.9569231794016635) -- (1.3170044215708698,3.979231427739324);
\draw (1.3170044215708698,3.979231427739324) -- (1.3235044001426775,4.000837531708705);
\draw (1.3235044001426775,4.000837531708705) -- (1.330004378714485,4.021793697669571);
\draw (1.330004378714485,4.021793697669571) -- (1.3365043572862927,4.042145727192177);
\draw (1.3365043572862927,4.042145727192177) -- (1.3430043358581003,4.061934069501075);
\draw (1.3430043358581003,4.061934069501075) -- (1.349504314429908,4.081194661177232);
\draw (1.349504314429908,4.081194661177232) -- (1.3560042930017155,4.099959603463128);
\draw (1.3560042930017155,4.099959603463128) -- (1.3625042715735232,4.118257714008684);
\draw (1.3625042715735232,4.118257714008684) -- (1.3690042501453308,4.1361149804061785);
\draw (1.3690042501453308,4.1361149804061785) -- (1.3755042287171384,4.153554936086854);
\draw (1.3755042287171384,4.153554936086854) -- (1.382004207288946,4.170598974242948);
\draw (1.382004207288946,4.170598974242948) -- (1.3885041858607536,4.187266611834141);
\draw (1.3885041858607536,4.187266611834141) -- (1.3950041644325613,4.203575713057518);
\draw (1.3950041644325613,4.203575713057518) -- (1.4015041430043689,4.219542679644946);
\draw (1.4015041430043689,4.219542679644946) -- (1.4080041215761765,4.235182613820449);
\draw (1.4080041215761765,4.235182613820449) -- (1.4145041001479841,4.25050945857495);
\draw (1.4145041001479841,4.25050945857495) -- (1.4210040787197917,4.265536119005714);
\draw (1.4210040787197917,4.265536119005714) -- (1.4275040572915993,4.280274567757051);
\draw (1.4275040572915993,4.280274567757051) -- (1.434004035863407,4.2947359370392375);
\draw (1.434004035863407,4.2947359370392375) -- (1.4405040144352146,4.308930599258847);
\draw (1.4405040144352146,4.308930599258847) -- (1.4470039930070222,4.322868237939154);
\draw (1.4470039930070222,4.322868237939154) -- (1.4535039715788298,4.336557910324303);
\draw (1.4535039715788298,4.336557910324303) -- (1.4600039501506374,4.350008102830357);
\draw (1.4600039501506374,4.350008102830357) -- (1.466503928722445,4.363226780318711);
\draw (1.466503928722445,4.363226780318711) -- (1.4730039072942527,4.3762214300138265);
\draw (1.4730039072942527,4.3762214300138265) -- (1.4795038858660603,4.388999100760957);
\draw (1.4795038858660603,4.388999100760957) -- (1.486003864437868,4.401566438215095);
\draw (1.486003864437868,4.401566438215095) -- (1.4925038430096755,4.413929716465688);
\draw (1.4925038430096755,4.413929716465688) -- (1.4990038215814832,4.426094866529251);
\draw (1.4990038215814832,4.426094866529251) -- (1.5055038001532908,4.438067502081406);
\draw (1.5055038001532908,4.438067502081406) -- (1.5120037787250984,4.449852942748839);
\draw (1.5120037787250984,4.449852942748839) -- (1.518503757296906,4.461456235238545);
\draw (1.518503757296906,4.461456235238545) -- (1.5250037358687136,4.472882172545217);
\draw (1.5250037358687136,4.472882172545217) -- (1.5315037144405212,4.48413531144652);
\draw (1.5315037144405212,4.48413531144652) -- (1.5380036930123289,4.495219988469436);
\draw (1.5380036930123289,4.495219988469436) -- (1.5445036715841365,4.506140334488128);
\draw (1.5445036715841365,4.506140334488128) -- (1.551003650155944,4.516900288094231);
\draw (1.551003650155944,4.516900288094231) -- (1.5575036287277517,4.52750360786361);
\draw (1.5575036287277517,4.52750360786361) -- (1.5640036072995593,4.537953883629097);
\draw (1.5640036072995593,4.537953883629097) -- (1.570503585871367,4.548254546856036);
\draw (1.570503585871367,4.548254546856036) -- (1.5770035644431746,4.558408880206569);
\draw (1.5770035644431746,4.558408880206569) -- (1.5835035430149822,4.568420026368959);
\draw (1.5835035430149822,4.568420026368959) -- (1.5900035215867898,4.578290996219948);
\draw (1.5900035215867898,4.578290996219948) -- (1.5965035001585974,4.588024676380792);
\draw (1.5965035001585974,4.588024676380792) -- (1.603003478730405,4.597623836221222);
\draw (1.603003478730405,4.597623836221222) -- (1.6095034573022127,4.607091134359916);
\draw (1.6095034573022127,4.607091134359916) -- (1.6160034358740203,4.616429124705093);
\draw (1.6160034358740203,4.616429124705093) -- (1.622503414445828,4.625640262074424);
\draw (1.622503414445828,4.625640262074424) -- (1.6290033930176355,4.6347269074296);
\draw (1.6290033930176355,4.6347269074296) -- (1.6355033715894431,4.643691332757377);
\draw (1.6355033715894431,4.643691332757377) -- (1.6420033501612508,4.652535725625919);
\draw (1.6420033501612508,4.652535725625919) -- (1.6485033287330584,4.661262193442475);
\draw (1.6485033287330584,4.661262193442475) -- (1.655003307304866,4.6698727674360025);
\draw (1.655003307304866,4.6698727674360025) -- (1.6615032858766736,4.678369406386179);
\draw (1.6615032858766736,4.678369406386179) -- (1.6680032644484812,4.68675400011829);
\draw (1.6680032644484812,4.68675400011829) -- (1.6745032430202889,4.695028372781724);
\draw (1.6745032430202889,4.695028372781724) -- (1.6810032215920965,4.703194285928253);
\draw (1.6810032215920965,4.703194285928253) -- (1.687503200163904,4.711253441404844);
\draw (1.687503200163904,4.711253441404844) -- (1.6940031787357117,4.719207484074498);
\draw (1.6940031787357117,4.719207484074498) -- (1.7005031573075193,4.72705800437746);
\draw (1.7005031573075193,4.72705800437746) -- (1.707003135879327,4.734806540744098);
\draw (1.707003135879327,4.734806540744098) -- (1.7135031144511346,4.742454581869822);
\draw (1.7135031144511346,4.742454581869822) -- (1.7200030930229422,4.750003568861569);
\draw (1.7200030930229422,4.750003568861569) -- (1.7265030715947498,4.757454897264612);
\draw (1.7265030715947498,4.757454897264612) -- (1.7330030501665574,4.764809918977731);
\draw (1.7330030501665574,4.764809918977731) -- (1.739503028738365,4.772069944064189);
\draw (1.739503028738365,4.772069944064189) -- (1.7460030073101727,4.779236242465345);
\draw (1.7460030073101727,4.779236242465345) -- (1.7525029858819803,4.786310045623214);
\draw (1.7525029858819803,4.786310045623214) -- (1.759002964453788,4.793292548017835);
\draw (1.759002964453788,4.793292548017835) -- (1.7655029430255955,4.800184908624821);
\draw (1.7655029430255955,4.800184908624821) -- (1.7720029215974031,4.806988252298104);
\draw (1.7720029215974031,4.806988252298104) -- (1.7785029001692108,4.813703671082503);
\draw (1.7785029001692108,4.813703671082503) -- (1.7850028787410184,4.82033222546041);
\draw (1.7850028787410184,4.82033222546041) -- (1.791502857312826,4.82687494553658);
\draw (1.791502857312826,4.82687494553658) -- (1.7980028358846336,4.833332832164735);
\draw (1.7980028358846336,4.833332832164735) -- (1.8045028144564412,4.839706858019414);
\draw (1.8045028144564412,4.839706858019414) -- (1.8110027930282488,4.8459979686162935);
\draw (1.8110027930282488,4.8459979686162935) -- (1.8175027716000565,4.852207083283934);
\draw (1.8175027716000565,4.852207083283934) -- (1.824002750171864,4.8583350960897755);
\draw (1.824002750171864,4.8583350960897755) -- (1.8305027287436717,4.864382876722928);
\draw (1.8305027287436717,4.864382876722928) -- (1.8370027073154793,4.870351271336222);
\draw (1.8370027073154793,4.870351271336222) -- (1.843502685887287,4.876241103349757);
\draw (1.843502685887287,4.876241103349757) -- (1.8500026644590946,4.882053174218065);
\draw (1.8500026644590946,4.882053174218065) -- (1.8565026430309022,4.887788264162873);
\draw (1.8565026430309022,4.887788264162873) -- (1.8630026216027098,4.893447132873323);
\draw (1.8630026216027098,4.893447132873323) -- (1.8695026001745174,4.899030520175352);
\draw (1.8695026001745174,4.899030520175352) -- (1.876002578746325,4.904539146671905);
\draw (1.876002578746325,4.904539146671905) -- (1.8825025573181327,4.909973714355456);
\draw (1.8825025573181327,4.909973714355456) -- (1.8890025358899403,4.915334907194315);
\draw (1.8890025358899403,4.915334907194315) -- (1.895502514461748,4.92062339169402);
\draw (1.895502514461748,4.92062339169402) -- (1.9020024930335555,4.9258398174351115);
\draw (1.9020024930335555,4.9258398174351115) -- (1.9085024716053631,4.930984817588448);
\draw (1.9085024716053631,4.930984817588448) -- (1.9150024501771707,4.936059009409208);
\draw (1.9150024501771707,4.936059009409208) -- (1.9215024287489784,4.9410629947106);
\draw (1.9215024287489784,4.9410629947106) -- (1.928002407320786,4.945997360318287);
\draw (1.928002407320786,4.945997360318287) -- (1.9345023858925936,4.9508626785064624);
\draw (1.9345023858925936,4.9508626785064624) -- (1.9410023644644012,4.955659507416428);
\draw (1.9410023644644012,4.955659507416428) -- (1.9475023430362088,4.960388391458546);
\draw (1.9475023430362088,4.960388391458546) -- (1.9540023216080165,4.965049861698306);
\draw (1.9540023216080165,4.965049861698306) -- (1.960502300179824,4.969644436227274);
\draw (1.960502300179824,4.969644436227274) -- (1.9670022787516317,4.974172620519606);
\draw (1.9670022787516317,4.974172620519606) -- (1.9735022573234393,4.978634907774789);
\draw (1.9735022573234393,4.978634907774789) -- (1.980002235895247,4.983031779247243);
\draw (1.980002235895247,4.983031779247243) -- (1.9865022144670546,4.987363704563354);
\draw (1.9865022144670546,4.987363704563354) -- (1.9930021930388622,4.991631142026515);
\draw (1.9930021930388622,4.991631142026515) -- (1.9995021716106698,4.995834538910694);
\draw (1.9995021716106698,4.995834538910694) -- (2.006002150182477,4.999974331743033);
\draw (2.006002150182477,4.999974331743033) -- (2.0125021287542846,5.004050946575946);
\draw (2.0125021287542846,5.004050946575946) -- (2.019002107326092,5.008064799249181);
\draw (2.019002107326092,5.008064799249181) -- (2.0255020858978994,5.012016295642251);
\draw (2.0255020858978994,5.012016295642251) -- (2.0320020644697068,5.015905831917661);
\draw (2.0320020644697068,5.015905831917661) -- (2.038502043041514,5.019733794755297);
\draw (2.038502043041514,5.019733794755297) -- (2.0450020216133216,5.023500561578353);
\draw (2.0450020216133216,5.023500561578353) -- (2.051502000185129,5.027206500771138);
\draw (2.051502000185129,5.027206500771138) -- (2.0580019787569364,5.030851971889094);
\draw (2.0580019787569364,5.030851971889094) -- (2.0645019573287438,5.034437325861335);
\draw (2.0645019573287438,5.034437325861335) -- (2.071001935900551,5.037962905186011);
\draw (2.071001935900551,5.037962905186011) -- (2.0775019144723585,5.041429044118769);
\draw (2.0775019144723585,5.041429044118769) -- (2.084001893044166,5.044836068854595);
\draw (2.084001893044166,5.044836068854595) -- (2.0905018716159733,5.048184297703278);
\draw (2.0905018716159733,5.048184297703278) -- (2.0970018501877807,5.0514740412587535);
\draw (2.0970018501877807,5.0514740412587535) -- (2.103501828759588,5.054705602562548);
\draw (2.103501828759588,5.054705602562548) -- (2.1100018073313955,5.057879277261556);
\draw (2.1100018073313955,5.057879277261556) -- (2.116501785903203,5.060995353760351);
\draw (2.116501785903203,5.060995353760351) -- (2.1230017644750103,5.06405411336824);
\draw (2.1230017644750103,5.06405411336824) -- (2.1295017430468177,5.067055830441244);
\draw (2.1295017430468177,5.067055830441244) -- (2.136001721618625,5.070000772519196);
\draw (2.136001721618625,5.070000772519196) -- (2.1425017001904325,5.07288920045813);
\draw (2.1425017001904325,5.07288920045813) -- (2.14900167876224,5.07572136855812);
\draw (2.14900167876224,5.07572136855812) -- (2.1555016573340473,5.078497524686735);
\draw (2.1555016573340473,5.078497524686735) -- (2.1620016359058547,5.081217910398262);
\draw (2.1620016359058547,5.081217910398262) -- (2.168501614477662,5.083882761048832);
\draw (2.168501614477662,5.083882761048832) -- (2.1750015930494695,5.0864923059075995);
\draw (2.1750015930494695,5.0864923059075995) -- (2.181501571621277,5.089046768264103);
\draw (2.181501571621277,5.089046768264103) -- (2.1880015501930843,5.091546365531931);
\draw (2.1880015501930843,5.091546365531931) -- (2.1945015287648917,5.093991309348809);
\draw (2.1945015287648917,5.093991309348809) -- (2.201001507336699,5.096381805673233);
\draw (2.201001507336699,5.096381805673233) -- (2.2075014859085065,5.098718054877757);
\draw (2.2075014859085065,5.098718054877757) -- (2.214001464480314,5.101000251839031);
\draw (2.214001464480314,5.101000251839031) -- (2.2205014430521213,5.103228586024695);
\draw (2.2205014430521213,5.103228586024695) -- (2.2270014216239287,5.105403241577234);
\draw (2.2270014216239287,5.105403241577234) -- (2.233501400195736,5.107524397394866);
\draw (2.233501400195736,5.107524397394866) -- (2.2400013787675435,5.109592227209571);
\draw (2.2400013787675435,5.109592227209571) -- (2.246501357339351,5.1116068996623305);
\draw (2.246501357339351,5.1116068996623305) -- (2.2530013359111583,5.113568578375659);
\draw (2.2530013359111583,5.113568578375659) -- (2.2595013144829657,5.115477422023517);
\draw (2.2595013144829657,5.115477422023517) -- (2.266001293054773,5.117333584398657);
\draw (2.266001293054773,5.117333584398657) -- (2.2725012716265804,5.119137214477494);
\draw (2.2725012716265804,5.119137214477494) -- (2.279001250198388,5.120888456482546);
\draw (2.279001250198388,5.120888456482546) -- (2.2855012287701952,5.1225874499425315);
\draw (2.2855012287701952,5.1225874499425315) -- (2.2920012073420026,5.124234329750159);
\draw (2.2920012073420026,5.124234329750159) -- (2.29850118591381,5.12582922621769);
\draw (2.29850118591381,5.12582922621769) -- (2.3050011644856174,5.127372265130314);
\draw (2.3050011644856174,5.127372265130314) -- (2.311501143057425,5.128863567797396);
\draw (2.311501143057425,5.128863567797396) -- (2.318001121629232,5.130303251101645);
\draw (2.318001121629232,5.130303251101645) -- (2.3245011002010396,5.131691427546253);
\draw (2.3245011002010396,5.131691427546253) -- (2.331001078772847,5.133028205300047);
\draw (2.331001078772847,5.133028205300047) -- (2.3375010573446544,5.134313688240696);
\draw (2.3375010573446544,5.134313688240696) -- (2.344001035916462,5.135547975996028);
\draw (2.344001035916462,5.135547975996028) -- (2.350501014488269,5.136731163983469);
\draw (2.350501014488269,5.136731163983469) -- (2.3570009930600766,5.13786334344767);
\draw (2.3570009930600766,5.13786334344767) -- (2.363500971631884,5.138944601496342);
\draw (2.363500971631884,5.138944601496342) -- (2.3700009502036914,5.139975021134335);
\draw (2.3700009502036914,5.139975021134335) -- (2.376500928775499,5.1409546812959865);
\draw (2.376500928775499,5.1409546812959865) -- (2.383000907347306,5.141883656875789);
\draw (2.383000907347306,5.141883656875789) -- (2.3895008859191136,5.142762018757379);
\draw (2.3895008859191136,5.142762018757379) -- (2.396000864490921,5.143589833840894);
\draw (2.396000864490921,5.143589833840894) -- (2.4025008430627284,5.144367165068707);
\draw (2.4025008430627284,5.144367165068707) -- (2.4090008216345358,5.145094071449584);
\draw (2.4090008216345358,5.145094071449584) -- (2.415500800206343,5.145770608081254);
\draw (2.415500800206343,5.145770608081254) -- (2.4220007787781506,5.146396826171445);
\draw (2.4220007787781506,5.146396826171445) -- (2.428500757349958,5.146972773057378);
\draw (2.428500757349958,5.146972773057378) -- (2.4350007359217654,5.147498492223757);
\draw (2.4350007359217654,5.147498492223757) -- (2.4415007144935728,5.14797402331926);
\draw (2.4415007144935728,5.14797402331926) -- (2.44800069306538,5.148399402171541);
\draw (2.44800069306538,5.148399402171541) -- (2.4545006716371875,5.148774660800774);
\draw (2.4545006716371875,5.148774660800774) -- (2.461000650208995,5.149099827431733);
\draw (2.461000650208995,5.149099827431733) -- (2.4675006287808023,5.149374926504424);
\draw (2.4675006287808023,5.149374926504424) -- (2.4740006073526097,5.149599978683291);
\draw (2.4740006073526097,5.149599978683291) -- (2.480500585924417,5.1497750008649765);
\draw (2.480500585924417,5.1497750008649765) -- (2.4870005644962245,5.149900006184677);
\draw (2.4870005644962245,5.149900006184677) -- (2.493500543068032,5.149975004021072);
\draw (2.493500543068032,5.149975004021072) -- (2.5000005216398393,5.149999999999839);
\draw (2.5000005216398393,5.149999999999839) -- (2.5065005002116467,5.14997499599577);
\draw (2.5065005002116467,5.14997499599577) -- (2.513000478783454,5.149899990133472);
\draw (2.513000478783454,5.149899990133472) -- (2.5195004573552615,5.149774976786662);
\draw (2.5195004573552615,5.149774976786662) -- (2.526000435927069,5.149599946576061);
\draw (2.526000435927069,5.149599946576061) -- (2.5325004144988763,5.1493748863658695);
\draw (2.5325004144988763,5.1493748863658695) -- (2.5390003930706837,5.149099779258838);
\draw (2.5390003930706837,5.149099779258838) -- (2.545500371642491,5.148774604589922);
\draw (2.545500371642491,5.148774604589922) -- (2.5520003502142985,5.148399337918503);
\draw (2.5520003502142985,5.148399337918503) -- (2.558500328786106,5.147973951019203);
\draw (2.558500328786106,5.147973951019203) -- (2.5650003073579133,5.147498411871236);
\draw (2.5650003073579133,5.147498411871236) -- (2.5715002859297207,5.146972684646335);
\draw (2.5715002859297207,5.146972684646335) -- (2.578000264501528,5.14639672969521);
\draw (2.578000264501528,5.14639672969521) -- (2.5845002430733355,5.145770503532539);
\draw (2.5845002430733355,5.145770503532539) -- (2.591000221645143,5.14509395882048);
\draw (2.591000221645143,5.14509395882048) -- (2.5975002002169503,5.144367044350684);
\draw (2.5975002002169503,5.144367044350684) -- (2.6040001787887577,5.143589705024794);
\draw (2.6040001787887577,5.143589705024794) -- (2.610500157360565,5.142761881833417);
\draw (2.610500157360565,5.142761881833417) -- (2.6170001359323725,5.141883511833544);
\draw (2.6170001359323725,5.141883511833544) -- (2.62350011450418,5.140954528124403);
\draw (2.62350011450418,5.140954528124403) -- (2.6300000930759873,5.139974859821713);
\draw (2.6300000930759873,5.139974859821713) -- (2.6365000716477947,5.138944432030339);
\draw (2.6365000716477947,5.138944432030339) -- (2.643000050219602,5.13786316581529);
\draw (2.643000050219602,5.13786316581529) -- (2.6495000287914094,5.136730978171061);
\draw (2.6495000287914094,5.136730978171061) -- (2.656000007363217,5.135547781989279);
\draw (2.656000007363217,5.135547781989279) -- (2.6624999859350242,5.134313486024627);
\draw (2.6624999859350242,5.134313486024627) -- (2.6689999645068316,5.133027994859006);
\draw (2.6689999645068316,5.133027994859006) -- (2.675499943078639,5.13169120886391);
\draw (2.675499943078639,5.13169120886391) -- (2.6819999216504464,5.130303024160984);
\draw (2.6819999216504464,5.130303024160984) -- (2.688499900222254,5.128863332580708);
\draw (2.688499900222254,5.128863332580708) -- (2.694999878794061,5.127372021619192);
\draw (2.694999878794061,5.127372021619192) -- (2.7014998573658686,5.125828974393019);
\draw (2.7014998573658686,5.125828974393019) -- (2.707999835937676,5.124234069592109);
\draw (2.707999835937676,5.124234069592109) -- (2.7144998145094834,5.1225871814305535);
\draw (2.7144998145094834,5.1225871814305535) -- (2.720999793081291,5.120888179595357);
\draw (2.720999793081291,5.120888179595357) -- (2.727499771653098,5.119136929193073);
\draw (2.727499771653098,5.119136929193073) -- (2.7339997502249056,5.1173332906942335);
\draw (2.7339997502249056,5.1173332906942335) -- (2.740499728796713,5.115477119875562);
\draw (2.740499728796713,5.115477119875562) -- (2.7469997073685204,5.113568267759877);
\draw (2.7469997073685204,5.113568267759877) -- (2.753499685940328,5.111606580553648);
\draw (2.753499685940328,5.111606580553648) -- (2.759999664512135,5.1095918995821235);
\draw (2.759999664512135,5.1095918995821235) -- (2.7664996430839426,5.107524061221989);
\draw (2.7664996430839426,5.107524061221989) -- (2.77299962165575,5.105402896831455);
\draw (2.77299962165575,5.105402896831455) -- (2.7794996002275574,5.103228232677716);
\draw (2.7794996002275574,5.103228232677716) -- (2.7859995787993648,5.100999889861719);
\draw (2.7859995787993648,5.100999889861719) -- (2.792499557371172,5.098717684240132);
\draw (2.792499557371172,5.098717684240132) -- (2.7989995359429796,5.0963814263444505);
\draw (2.7989995359429796,5.0963814263444505) -- (2.805499514514787,5.093990921297152);
\draw (2.805499514514787,5.093990921297152) -- (2.8119994930865944,5.091545968724794);
\draw (2.8119994930865944,5.091545968724794) -- (2.8184994716584018,5.0890463626679745);
\draw (2.8184994716584018,5.0890463626679745) -- (2.824999450230209,5.086491891488049);
\draw (2.824999450230209,5.086491891488049) -- (2.8314994288020165,5.083882337770496);
\draw (2.8314994288020165,5.083882337770496) -- (2.837999407373824,5.081217478224825);
\draw (2.837999407373824,5.081217478224825) -- (2.8444993859456313,5.078497083580911);
\draw (2.8444993859456313,5.078497083580911) -- (2.8509993645174387,5.075720918481637);
\draw (2.8509993645174387,5.075720918481637) -- (2.857499343089246,5.072888741371715);
\draw (2.857499343089246,5.072888741371715) -- (2.8639993216610535,5.070000304382553);
\draw (2.8639993216610535,5.070000304382553) -- (2.870499300232861,5.067055353213033);
\draw (2.870499300232861,5.067055353213033) -- (2.8769992788046683,5.06405362700606);
\draw (2.8769992788046683,5.06405362700606) -- (2.8834992573764757,5.060994858220717);
\draw (2.8834992573764757,5.060994858220717) -- (2.889999235948283,5.057878772499876);
\draw (2.889999235948283,5.057878772499876) -- (2.8964992145200905,5.054705088533106);
\draw (2.8964992145200905,5.054705088533106) -- (2.902999193091898,5.051473517914682);
\draw (2.902999193091898,5.051473517914682) -- (2.9094991716637053,5.048183764996535);
\draw (2.9094991716637053,5.048183764996535) -- (2.9159991502355127,5.044835526735939);
\draw (2.9159991502355127,5.044835526735939) -- (2.92249912880732,5.041428492537733);
\draw (2.92249912880732,5.041428492537733) -- (2.9289991073791275,5.037962344090876);
\draw (2.9289991073791275,5.037962344090876) -- (2.935499085950935,5.034436755199104);
\draw (2.935499085950935,5.034436755199104) -- (2.9419990645227423,5.0308513916054585);
\draw (2.9419990645227423,5.0308513916054585) -- (2.9484990430945497,5.027205910810455);
\draw (2.9484990430945497,5.027205910810455) -- (2.954999021666357,5.02349996188361);
\draw (2.954999021666357,5.02349996188361) -- (2.9614990002381645,5.01973318526808);
\draw (2.9614990002381645,5.01973318526808) -- (2.967998978809972,5.015905212578122);
\draw (2.967998978809972,5.015905212578122) -- (2.9744989573817793,5.012015666389075);
\draw (2.9744989573817793,5.012015666389075) -- (2.9809989359535867,5.008064160019549);
\draw (2.9809989359535867,5.008064160019549) -- (2.987498914525394,5.004050297305499);
\draw (2.987498914525394,5.004050297305499) -- (2.9939988930972015,4.999973672365831);
\draw (2.9939988930972015,4.999973672365831) -- (3.000498871669009,4.995833869359181);
\draw (3.000498871669009,4.995833869359181) -- (3.0069988502408163,4.991630462231473);
\draw (3.0069988502408163,4.991630462231473) -- (3.0134988288126237,4.987363014453864);
\draw (3.0134988288126237,4.987363014453864) -- (3.019998807384431,4.9830310787506376);
\draw (3.019998807384431,4.9830310787506376) -- (3.0264987859562384,4.978634196816607);
\draw (3.0264987859562384,4.978634196816607) -- (3.032998764528046,4.974171899023544);
\draw (3.032998764528046,4.974171899023544) -- (3.0394987430998532,4.969643704115139);
\draw (3.0394987430998532,4.969643704115139) -- (3.0459987216716606,4.9650491188899615);
\draw (3.0459987216716606,4.9650491188899615) -- (3.052498700243468,4.960387637871856);
\draw (3.052498700243468,4.960387637871856) -- (3.0589986788152754,4.955658742967206);
\draw (3.0589986788152754,4.955658742967206) -- (3.065498657387083,4.950861903108407);
\draw (3.065498657387083,4.950861903108407) -- (3.0719986359588902,4.945996573882928);
\draw (3.0719986359588902,4.945996573882928) -- (3.0784986145306976,4.941062197147227);
\draw (3.0784986145306976,4.941062197147227) -- (3.084998593102505,4.936058200624814);
\draw (3.084998593102505,4.936058200624814) -- (3.0914985716743124,4.930983997487652);
\draw (3.0914985716743124,4.930983997487652) -- (3.09799855024612,4.925838985920092);
\draw (3.09799855024612,4.925838985920092) -- (3.104498528817927,4.920622548664442);
\draw (3.104498528817927,4.920622548664442) -- (3.1109985073897346,4.915334052547244);
\draw (3.1109985073897346,4.915334052547244) -- (3.117498485961542,4.90997284798529);
\draw (3.117498485961542,4.90997284798529) -- (3.1239984645333494,4.904538268470278);
\draw (3.1239984645333494,4.904538268470278) -- (3.130498443105157,4.899029630031054);
\draw (3.130498443105157,4.899029630031054) -- (3.136998421676964,4.8934462306722);
\draw (3.136998421676964,4.8934462306722) -- (3.1434984002487716,4.887787349787741);
\draw (3.1434984002487716,4.887787349787741) -- (3.149998378820579,4.882052247548598);
\draw (3.149998378820579,4.882052247548598) -- (3.1564983573923864,4.876240164262395);
\draw (3.1564983573923864,4.876240164262395) -- (3.1629983359641938,4.870350319704048);
\draw (3.1629983359641938,4.870350319704048) -- (3.169498314536001,4.8643819124155625);
\draw (3.169498314536001,4.8643819124155625) -- (3.1759982931078086,4.858334118973252);
\draw (3.1759982931078086,4.858334118973252) -- (3.182498271679616,4.852206093220573);
\draw (3.182498271679616,4.852206093220573) -- (3.1889982502514234,4.845996965464571);
\draw (3.1889982502514234,4.845996965464571) -- (3.1954982288232308,4.839705841633826);
\draw (3.1954982288232308,4.839705841633826) -- (3.201998207395038,4.833331802395646);
\draw (3.201998207395038,4.833331802395646) -- (3.2084981859668456,4.826873902230078);
\draw (3.2084981859668456,4.826873902230078) -- (3.214998164538653,4.820331168458141);
\draw (3.214998164538653,4.820331168458141) -- (3.2214981431104603,4.813702600221507);
\draw (3.2214981431104603,4.813702600221507) -- (3.2279981216822677,4.806987167410637);
\draw (3.2279981216822677,4.806987167410637) -- (3.234498100254075,4.800183809538169);
\draw (3.234498100254075,4.800183809538169) -- (3.2409980788258825,4.79329143455412);
\draw (3.2409980788258825,4.79329143455412) -- (3.24749805739769,4.786308917599188);
\draw (3.24749805739769,4.786308917599188) -- (3.2539980359694973,4.779235099692174);
\draw (3.2539980359694973,4.779235099692174) -- (3.2604980145413047,4.772068786347229);
\draw (3.2604980145413047,4.772068786347229) -- (3.266997993113112,4.7648087461162785);
\draw (3.266997993113112,4.7648087461162785) -- (3.2734979716849195,4.757453709051661);
\draw (3.2734979716849195,4.757453709051661) -- (3.279997950256727,4.750002365083536);
\draw (3.279997950256727,4.750002365083536) -- (3.2864979288285343,4.742453362306267);
\draw (3.2864979288285343,4.742453362306267) -- (3.2929979074003417,4.734805305167425);
\draw (3.2929979074003417,4.734805305167425) -- (3.299497885972149,4.727056752552599);
\draw (3.299497885972149,4.727056752552599) -- (3.3059978645439565,4.719206215758573);
\draw (3.3059978645439565,4.719206215758573) -- (3.312497843115764,4.711252156346818);
\draw (3.312497843115764,4.711252156346818) -- (3.3189978216875713,4.703192983868551);
\draw (3.3189978216875713,4.703192983868551) -- (3.3254978002593787,4.695027053451833);
\draw (3.3254978002593787,4.695027053451833) -- (3.331997778831186,4.686752663240336);
\draw (3.331997778831186,4.686752663240336) -- (3.3384977574029935,4.678368051672481);
\draw (3.3384977574029935,4.678368051672481) -- (3.344997735974801,4.66987139458859);
\draw (3.344997735974801,4.66987139458859) -- (3.3514977145466083,4.661260802152576);
\draw (3.3514977145466083,4.661260802152576) -- (3.3579976931184157,4.65253431557342);
\draw (3.3579976931184157,4.65253431557342) -- (3.364497671690223,4.643689903610244);
\draw (3.364497671690223,4.643689903610244) -- (3.3709976502620305,4.63472545884326);
\draw (3.3709976502620305,4.63472545884326) -- (3.377497628833838,4.625638793691103);
\draw (3.377497628833838,4.625638793691103) -- (3.3839976074056453,4.61642763615311);
\draw (3.3839976074056453,4.61642763615311) -- (3.3904975859774527,4.607089625252927);
\draw (3.3904975859774527,4.607089625252927) -- (3.39699756454926,4.597622306157406);
\draw (3.39699756454926,4.597622306157406) -- (3.4034975431210674,4.588023124941986);
\draw (3.4034975431210674,4.588023124941986) -- (3.409997521692875,4.578289422970709);
\draw (3.409997521692875,4.578289422970709) -- (3.4164975002646822,4.568418430855558);
\draw (3.4164975002646822,4.568418430855558) -- (3.4229974788364896,4.558407261955908);
\draw (3.4229974788364896,4.558407261955908) -- (3.429497457408297,4.548252905374486);
\draw (3.429497457408297,4.548252905374486) -- (3.4359974359801044,4.537952218401232);
\draw (3.4359974359801044,4.537952218401232) -- (3.442497414551912,4.527501918350847);
\draw (3.442497414551912,4.527501918350847) -- (3.4489973931237192,4.516898573733356);
\draw (3.4489973931237192,4.516898573733356) -- (3.4554973716955266,4.506138594689702);
\draw (3.4554973716955266,4.506138594689702) -- (3.461997350267334,4.495218222616052);
\draw (3.461997350267334,4.495218222616052) -- (3.4684973288391414,4.484133518890924);
\draw (3.4684973288391414,4.484133518890924) -- (3.474997307410949,4.472880352608248);
\draw (3.474997307410949,4.472880352608248) -- (3.481497285982756,4.461454387206898);
\draw (3.481497285982756,4.461454387206898) -- (3.4879972645545636,4.4498510658726165);
\draw (3.4879972645545636,4.4498510658726165) -- (3.494497243126371,4.438065595571432);
\draw (3.494497243126371,4.438065595571432) -- (3.5009972216981784,4.426092929554135);
\draw (3.5009972216981784,4.426092929554135) -- (3.507497200269986,4.413927748148595);
\draw (3.507497200269986,4.413927748148595) -- (3.513997178841793,4.401564437630182);
\draw (3.513997178841793,4.401564437630182) -- (3.5204971574136006,4.388997066929454);
\draw (3.5204971574136006,4.388997066929454) -- (3.526997135985408,4.376219361899688);
\draw (3.526997135985408,4.376219361899688) -- (3.5334971145572154,4.363224676823805);
\draw (3.5334971145572154,4.363224676823805) -- (3.5399970931290228,4.350005962789113);
\draw (3.5399970931290228,4.350005962789113) -- (3.54649707170083,4.336555732497755);
\draw (3.54649707170083,4.336555732497755) -- (3.5529970502726376,4.322866021008281);
\draw (3.5529970502726376,4.322866021008281) -- (3.559497028844445,4.308928341817114);
\draw (3.559497028844445,4.308928341817114) -- (3.5659970074162524,4.294733637584216);
\draw (3.5659970074162524,4.294733637584216) -- (3.5724969859880598,4.280272224680971);
\draw (3.5724969859880598,4.280272224680971) -- (3.578996964559867,4.265533730584785);
\draw (3.578996964559867,4.265533730584785) -- (3.5854969431316746,4.250507022957237);
\draw (3.5854969431316746,4.250507022957237) -- (3.591996921703482,4.235180129012068);
\draw (3.591996921703482,4.235180129012068) -- (3.5984969002752893,4.2195401434942745);
\draw (3.5984969002752893,4.2195401434942745) -- (3.6049968788470967,4.20357312323707);
\draw (3.6049968788470967,4.20357312323707) -- (3.611496857418904,4.18726396581966);
\draw (3.611496857418904,4.18726396581966) -- (3.6179968359907115,4.1705962692891845);
\draw (3.6179968359907115,4.1705962692891845) -- (3.624496814562519,4.153552169199354);
\draw (3.624496814562519,4.153552169199354) -- (3.6309967931343263,4.13611214830825);
\draw (3.6309967931343263,4.13611214830825) -- (3.6374967717061337,4.118254813102497);
\draw (3.6374967717061337,4.118254813102497) -- (3.643996750277941,4.099956629783616);
\draw (3.643996750277941,4.099956629783616) -- (3.6504967288497485,4.081191610337098);
\draw (3.6504967288497485,4.081191610337098) -- (3.656996707421556,4.061930936624701);
\draw (3.656996707421556,4.061930936624701) -- (3.6634966859933633,4.042142506835576);
\draw (3.6634966859933633,4.042142506835576) -- (3.6699966645651707,4.021790383722592);
\draw (3.6699966645651707,4.021790383722592) -- (3.676496643136978,4.0008341172744295);
\draw (3.676496643136978,4.0008341172744295) -- (3.6829966217087855,3.9792279049837798);
\draw (3.6829966217087855,3.9792279049837798) -- (3.689496600280593,3.9569195393644114);
\draw (3.689496600280593,3.9569195393644114) -- (3.6959965788524003,3.933849072801018);
\draw (3.6959965788524003,3.933849072801018) -- (3.7024965574242077,3.9099471009041435);
\draw (3.7024965574242077,3.9099471009041435) -- (3.708996535996015,3.885132521884909);
\draw (3.708996535996015,3.885132521884909) -- (3.7154965145678225,3.859309561906348);
\draw (3.7154965145678225,3.859309561906348) -- (3.72199649313963,3.8323637488974196);
\draw (3.72199649313963,3.8323637488974196) -- (3.7284964717114373,3.804156340889211);
\draw (3.7284964717114373,3.804156340889211) -- (3.7349964502832447,3.7745164146445136);
\draw (3.7349964502832447,3.7745164146445136) -- (3.741496428855052,3.743229287302954);
\draw (3.741496428855052,3.743229287302954) -- (3.7479964074268595,3.71001894948846);
\draw (3.7479964074268595,3.71001894948846) -- (3.754496385998667,3.674520220370009);
\draw (3.754496385998667,3.674520220370009) -- (3.7609963645704743,3.6362321469077323);
\draw (3.7609963645704743,3.6362321469077323) -- (3.7674963431422817,3.5944344058628435);
\draw (3.7674963431422817,3.5944344058628435) -- (3.773996321714089,3.548022842135311);
\draw (3.773996321714089,3.548022842135311) -- (3.7804963002858964,3.49514116811377);
\draw (3.7804963002858964,3.49514116811377) -- (3.786996278857704,3.4321748931266938);
\draw (3.786996278857704,3.4321748931266938) -- (3.7934962574295112,3.349807197102207);
\draw (3.7934962574295112,3.349807197102207) -- (3.7999962360013186,3.1548127991102097);
\draw [shift={(2.5027950853277097,5.374696262536255)}] plot[domain=4.183381775843862:5.240284126484102,variable=\t]({1.*2.575269704553821*cos(\t r)+0.*2.575269704553821*sin(\t r)},{0.*2.575269704553821*cos(\t r)+1.*2.575269704553821*sin(\t r)});
\draw [shift={(2.5001237865471433,0.5704550360140033)},dash pattern=on 1pt off 1pt]  plot[domain=1.1040201370666363:2.0364604580817427,variable=\t]({1.*2.88855160790422*cos(\t r)+0.*2.88855160790422*sin(\t r)},{0.*2.88855160790422*cos(\t r)+1.*2.88855160790422*sin(\t r)});
\draw [shift={(2.5000334865774017,1.8375589421755396)}] plot[domain=4.1132251718625135:5.311439659032727,variable=\t]({1.*0.7154022367760153*cos(\t r)+0.*0.7154022367760153*sin(\t r)},{0.*0.7154022367760153*cos(\t r)+1.*0.7154022367760153*sin(\t r)});
\draw [shift={(2.4999615685688013,0.566135119397119)},dash pattern=on 1pt off 1pt]  plot[domain=1.035685835982406:2.1057936877332493,variable=\t]({1.*0.7911926808946546*cos(\t r)+0.*0.7911926808946546*sin(\t r)},{0.*0.7911926808946546*cos(\t r)+1.*0.7911926808946546*sin(\t r)});
\draw [dash pattern=on 1pt off 1pt] (2.067707825203131,1.)-- (2.932292174796869,1.);
\draw (2.932292174796869,1.)-- (5.,1.);
\draw (3.8587247545003462,3.318281998069933) node[anchor=north west] {$r=r_0$};
\draw (2.977312919196823,1.4432951282203845) node[anchor=north west] {$r=r_1$};
\end{tikzpicture}
\caption{A non-ergodic DBG torus}
\end{figure}
\end{center}

\section{Non-ergodic Donnay-Burns-Gerber tori}
\begin{definition}
We say that a centrally symmetric cap $\mathscr{C}=\{r\leq r_1\}\subseteq \mathbb{R}^2$ is a \textbf{non-ergodic Donnay-Burns-Gerber (DBG) cap} if:

(a) $\mathscr{C}$ has two parallel geodesics $C_{r_0}$ and $C_{r_1}$, where $C_{r_i}:=\{r=r_i\}$ for $i=0,1$. 

(b) The Gaussian curvature is positive on $\{r\leq r_0\}$, negative at $C_{r_1}$, and strictly decreasing from center to boundary. 

If a torus contains a  non-ergodic DBG cap and outside the cap the Gaussian curvature is nonpositive, then we call it a \textbf{non-ergodic DBG torus}.

\end{definition}

\begin{lemma}\label{lem1}
The geodesic flow on a non-ergodic DBG torus has positive metric entropy.
\end{lemma}
\begin{proof}[Sketch of proof]
The proof is similar to the proof of Theorem 1.1 in \cite{BG}. By virtue of Clairaut's integral, any geodesic entering the cap $\mathscr{C}$ will go out of the cap.
 Let $c:[-T_1, T_1]\rightarrow \mathscr{C}$ be an arc-length parametrized geodesic with endpoints in $C_{r_1}$ such that $c(0)$ is the point of $c$ closest to the origin; suppose furthermore that $c(\pm T_2)$ lie in $C_{r_0}$, for some $0<T_2<T_1$. Let $J_S, J_C$ be two Jacobi fields on $c$ with $J_S(0)=0, J_S'(0)=1, J_C(0)=1, J_C'(0)=0$. Let $u_S=J_S'/J_S, u_C=J_C'/J_C$ and $K(t)$ be the Gaussian curvature at $c(t)$. Then both $u_S$ and $u_C$ satisfy the Riccati equation:
$$u'(t)+u(t)^2+K(t)=0$$

By imitating the proofs of Lemma 2.5 and Lemma 2.6 in \cite{BG}, we get

(A) $u_S(\pm T_1)=u_S(\pm T_2)=0.$ and $J_S(t)$ vanishes only at $t=0$. 

(B) There is a $\tau\in(0, T_2)$ such that $\lim_{t\rightarrow\tau^-}u_C(t)=-\infty.$

If a Jacobi field $J$ on $c$ satisfies $J'(-T_1)J(-T_1)\geq 0$ then $u:=J'/J$ satisfies the Riccati equation with $u(-T_1)\geq 0$. This means the graph of $u$ must lie above that of $u_S$. By (A) and (B) we have $u(T_1)\geq 0$. So the cone $J'J\geq0$ is preserved by the cap. 

\begin{center}
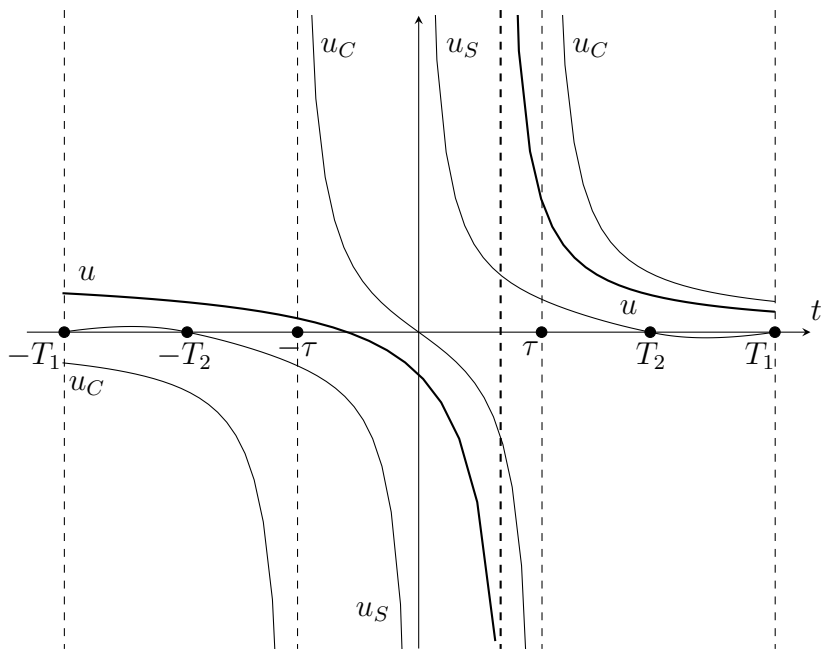
\begin{figure}[h]
\begin{tikzpicture}
\begin{axis}[width=12cm, height=10cm, axis x line=middle, axis y line=middle, xtick=\empty,  ytick=\empty, 
xmin=-2.2, xmax=2.2, ymin=-1, ymax=1]
\addplot[domain=-2: -1.3]{(x+1.3)*(x+2)/(4.2*x)};
\addplot[domain=-1.3: -0.05]{(x^3+1.3^3)*(x^3+2^3)/(200*x)};
\addplot[domain=0.05: 1.3]{(x^3-1.3^3)*(x^3-2^3)/(200*x)};
\addplot[domain=1.3: 2]{(x-1.3)*(x-2)/(4.2*x)};
\addplot[domain=-0.63: 0.63]{0.2*x/(x^2-0.49)};
\addplot[domain=0.77: 2]{0.1/(x-0.7)+0.02};
\addplot[domain=-2: -0.77]{0.1/(x+0.7)-0.02};
\addplot[thick, domain=-2: 0.43]{0.2/(x-0.6)+0.2};
\addplot[thick, domain=0.5: 2]{0.1/(x-0.45)};
\addplot[mark=*] coordinates {(-1.99, 0)};
\addplot[mark=*] coordinates {(2, 0)};
\addplot[mark=*] coordinates {(-1.3,0)};
\addplot[mark=*] coordinates {(1.3,0)};
\addplot[mark=*] coordinates {(-0.68,0)};
\addplot[mark=*] coordinates {(0.69,0)};
\end{axis}
\draw (10.5, 4.5) node{$t$};
\draw (4.6, 0.5) node{${u_S}$};
\draw (5.8, 8) node{${u_S}$};
\draw (0.8, 3.5) node{${u_C}$};
\draw (4.15, 8) node{${u_C}$};
\draw (7.5, 8) node{${u_C}$};
\draw (0.8, 5) node{$u$};
\draw (8, 4.55) node{$u$};
\draw (0.1, 3.9) node{${-T_1}$};
\draw (9.75, 3.9) node{${T_1}$};
\draw (2.1, 3.9) node{${-T_2}$};
\draw (3.6, 4) node{${-\tau}$};
\draw (6.7, 4) node{${\tau}$};
\draw (8.3, 3.9) node{${T_2}$};
\draw[dashed] (0.5, 8.5) -- (0.5,0);
\draw[dashed] (3.6, 8.5) -- (3.6,0);
\draw[thick, dashed] (6.3, 8.5) -- (6.3,0);
\draw[dashed] (6.85, 8.5) -- (6.85,0);
\draw[dashed] (9.95, 8.5) -- (9.95,0);
\end{tikzpicture}
\caption{Graphs of $u_S$, $u_C$ and $u$}
\end{figure}
\end{center}

By Poincar\'{e} recurrence theorem, almost every vector in $UT\mathscr{C}$ will come back infinitely many times.  For any geodesic $c$ entering the cap $\mathscr{C}$ at time $t_0$, when it returns to the cap again, say at time $t_1>t_0$, the image of the cone $\{J'(t_0)J(t_0)\geq 0\}$ under the translation will lie strictly in the interior of $\{J'(t_1)J(t_1)\geq 0\}$. By Wojkowski's cone field theory \cite{W}, the vectors with non-zero Lyapunov exponents form a set with positive Liouville measure. By Pesin's inequality (1) the geodesic flow has positive metric entropy.
\end{proof}

\section{Construction of a Non-ergodic DBG torus}
In this section we construct a conformal metric on $[-1,1]\times[-1,1]$ which is flat outside a disc and centrally symmetric inside the disc. More precisely we want to build  a function $g:[0,2]\rightarrow (0,1]$ such that the torus with conformal metric
$$ds^2=g(r)^2(dx^2+dy^2), \text{ where } r:=\sqrt{x^2+y^2}.\eqno (2)$$
is a non-ergodic DBG torus.

In order to get such a function $g$ we change our coordinate system to  geodesic polar coordinates. However before doing this we need some preliminary. 
\begin{definition}
We say a function $\rho: I\rightarrow \mathbb{R}$ is \textit{even} (resp. \textit{odd}) at a point $a\in I$ if all odd (resp. even) derivatives of $\rho$ vanish at $a$. 
\end{definition}

\begin{lemma}\label{lem2}
For any smooth function $\rho: \mathbb{R}_{\geq 0}\rightarrow \mathbb{R}_{\geq 0}$ which is odd at 0, $\rho'(0)=1$ and is positive except at 0, there exist smooth functions $g, l: \mathbb{R}_{\geq 0}\rightarrow \mathbb{R}_{\geq 0}$ such that $l$ is odd at 0, $l(0)=0, l'(r)=g(r),  g(0)=1, \rho(l(r))=rg(r)$, and $g$ is positive. 
\end{lemma}
\begin{proof}
Since
$$\rho=r\frac{dl}{dr}, $$
we have
$$\frac{dr}{r}=\frac{dl}{\rho}. \eqno(\ast)$$
Both sides of $(\ast)$ have singularity at 0. Since $\rho$ is odd at 0, $\rho'(0)=1$, for small $l$ we have
$$\frac{1}{\rho}=\frac{1}{l}\left(\frac{1}{1+\rho^{(3)}(0)l^2/6+o(l^3)}\right)=\frac{1}{l}\left(\frac{1}{1+l^2O(1)}\right)=\frac{1}{l}(1+l^2\tilde{\rho}(l))=\frac{1}{l}+l\tilde{\rho}(l),$$
where $\tilde{\rho}$ is a smooth function that is even at 0. We integrate both sides of $(\ast)$ regarding $r$ as a function of $l$ with $r(0)=0$. Then we get
$$\lim_{l\rightarrow 0}(\ln r-\ln l)=\lim_{l\rightarrow 0}\int_0^l s\tilde{\rho}(s)ds=0.$$
Therefore $\lim_{l\rightarrow 0}\ln(r/l)=0$ and $\ln(r/l)$ is even at 0. By direct computation, it is now easy to see that $r/l$ is even at 0. This implies that $r$ is odd at 0. From $(\ast)$ we have
$$\frac{d\ln r}{dl}=\frac{dr}{rdl}=\frac{1}{\rho}>0.$$
Therefore $\ln r(l)$ is strictly increasing and smooth, so is $r(l)$. By the Inverse Function Theorem there exists a smooth $l: \mathbb{R}_{\geq 0}\rightarrow \mathbb{R}_{\geq 0}$ which is the inverse function of $r(l)$. Moreover $l(0)=0$, $l'(0)=1$ and $l$ is odd at 0. Finally we define $g(r):=l'(r)$. It is clear that $g$ is even at 0 and positive.
\end{proof}
By Lemma \ref{lem2} we have only to find $\rho: \mathbb{R}_{\geq 0}\rightarrow \mathbb{R}_{\geq 0}$ with the following properties:

(i) $\rho$ satisfies the conditions in Lemma \ref{lem2};

(ii) $\rho'(l_0)=\rho'(l_1)=0$ for some $0<l_0<l_1$.

(iii) Let $K(l):=-\rho''(l)/\rho(l)$. Then $K(l)>0$ on $[0,l_0]$, $K(l_1)<0$.

(iv) $K'(l)<0$ on $[0,l_1]$.

(v) There exists $l_2>l_1$ such that $K(l)$ is negative on $[l_1,l_2)$ and $\rho'(l)=1$ for $l\geq l_2$.

Indeed once we have such a function $\rho$, by Lemma \ref{lem2} we have smooth functions $g,l:\mathbb{R}_{\geq 0}\rightarrow \mathbb{R}_{\geq 0}$ with $\rho(l(r))=rg(r)$ and $l(r)=\int_0^r g(t)dt$. Consider the metric defined by (2). By changing the coordinate system to geodesic polar coordinates, the metric becomes
$$ds^2=dl^2+\rho(l)^2d\theta^2.\eqno(3)$$

Note that $\rho'(l)=0$ iff the parallel at $l$ is a geodesic, and the Gaussian curvature is given by $K(l)=-\rho''(l)/\rho(l)$. Let $r_i:=l^{-1}(l_i)$ for $i=0,1,2$. (ii) implies (a) in the definition of a non-ergodic DBG cap, while (b) can be derived from (iii) and (iv). (v) guarantees the metric is negatively curved on the annulus $\{r_1<r<r_2\}$ and is flat outside $\{r=r_2\}$. So once $\rho$ satisfies (i)-(v), the torus with metric (3) will be a non-ergodic DBG torus.\\

Here is the construction of $\rho(l)$:

For any $a>0$, let $\lambda_1: \mathbb{R}_{\geq 0}\rightarrow [0,1]$  be a $C^{\infty}$ function with the properties that $\lambda_1\equiv 1$ on $[0,\frac{1}{\sqrt{5a}}]$ and $\lambda_1\equiv 0$ on $[\frac{1}{2\sqrt{a}},+\infty)$. Let $\lambda_2: \mathbb{R}_{\geq 0}\rightarrow [0,1]$ be another $C^{\infty}$ function with $\lambda_2 \equiv 0$ on $[0,\frac{1}{\sqrt{5a}}]\cup[\frac{1}{2\sqrt{a}},+\infty)$ and positive on $(\frac{1}{\sqrt{5a}}, \frac{1}{2\sqrt{a}})$. Define $\rho''(l)$ by 
$$\rho''(l)=\lambda_1(l)(-30al+200a^2l^3)+C(1-\lambda_1(l))\lambda_2(l),$$
where $C$ is a positive constant such that $\int_0^{\infty}\rho''(l)dl=0$. Notice that $\rho''(l)=0$ on $[\frac{1}{2\sqrt{a}},+\infty)$. Define $\rho$ by setting $\rho(0)=0, \rho'(0)=1$. Then $\rho(l)=l-5al^3+10a^2l^5$ on $[0,\frac{1}{\sqrt{5a}}]$ and $\rho'(l)\equiv 1$ for $l\geq \frac{1}{2\sqrt{a}}$. The graph of $\rho$ is shown in Figure 3.

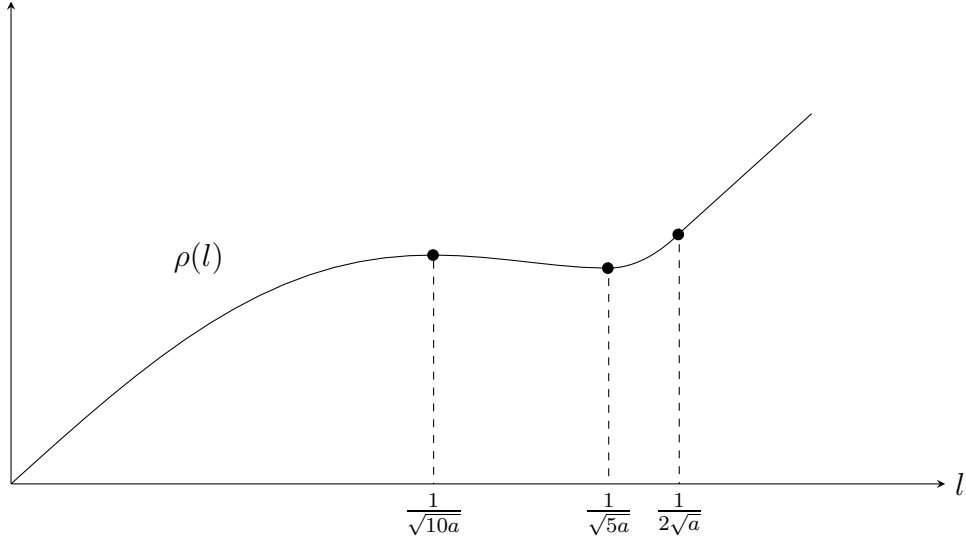
\begin{figure}
\begin{tikzpicture}
\begin{axis}[width=14cm, height=8cm, axis x line=bottom, axis y line=left, xtick=\empty,  ytick=\empty, xmin=0, xmax=0.7, ymin=0, ymax=0.4]
\addplot[domain=0: sqrt{0.2}]{(10*x^5-5*x^3+x)};
\addplot[domain=sqrt{0.2}: 0.5]{10*(x-sqrt{0.2})^2+2/(5*sqrt(5))};
\addplot[domain= 0.5: 0.6]{x+4-48*sqrt(5)/25};
\addplot[mark=*] coordinates {(sqrt(0.2), 0.4/sqrt 5)};
\addplot[mark=*] coordinates {(sqrt(0.1), 0.6/sqrt 10)};
\addplot[mark=*] coordinates {(0.5, 4.5-1.92*sqrt 5)};
\end{axis}
\draw (2.5, 3) node{$\rho(l)$};
\draw (17.77*sqrt{0.2}, -0.4) node{$\frac{1}{\sqrt{5a}}$};
\draw (17.77*sqrt{0.1}, -0.4) node{$\frac{1}{\sqrt{10a}}$};
\draw (17.77*0.5, -0.4) node{$\frac{1}{2\sqrt{a}}$};
\draw (17.77*0.71, 0) node{$l$};
\draw[dashed] (17.77*sqrt{0.2}, 2.95) -- (17.77*sqrt{0.2}, 0);
\draw[dashed] (17.77*sqrt{0.1}, 2.97) -- (17.77*sqrt{0.1}, 0);
\draw[dashed] (17.77*0.5, 3.4) -- (17.77*0.5, 0);
\end{tikzpicture}
\caption{Graph of $\rho$}
\end{figure}

It is easy to see that $\rho$ satisfies (i) and (ii) for $l_0=\frac{1}{\sqrt{10a}}$ and $l_1=\frac{1}{\sqrt{5a}}$. $\rho''(l)=-30al+200a^2l^3$ on $[0, l_1]$ and it is positive on $[l_1, \frac{1}{2\sqrt{a}})$, so $K(l)$ is positive on $\{l \leq \frac{1}{\sqrt{10a}}\}$ and negative in the annulus between $\{l= l_1\}$ and $\{l= \frac{1}{2\sqrt{a}}\}$. Hence $\rho$ satisfies (iii) and (v) for $l_2=\frac{1}{2\sqrt{a}}$. The last part to be verified is (iv). Since
$$K'(l)=-\frac{\rho\rho'''-\rho'\rho''}{\rho^2},$$
we only need to verify that $\rho\rho'''-\rho'\rho''=100a^2l^3(1+12al^2-40a^2l^4)$ is positive on $(0,\frac{1}{\sqrt{5a}}]$. This can be done by direct calculation. This finishes the construction.

\begin{remark}
The function $g$ constructed in this way is strictly decreasing on $[0,r_2]$ and constant for $r\geq r_2$ since 
$$\frac{d\rho}{dl}=\frac{d\rho}{dr}\frac{dr}{dl}=\frac{g+rg'}{g}=1+\frac{rg'}{g}$$
and $\rho'(l)<1$ on $(0,\frac{1}{2\sqrt{a}})$, $\rho'(l)=1$ for $l\geq \frac{1}{2\sqrt{a}}$. 
So the supremum of $g$ is $g(0)=1$. From Lemma \ref{lem2} we know that the lower bound is positive. 
\end{remark}
\begin{remark}\label{rem2}
If $g$ satisfies the condition that a torus with metric $g(r)^2(dx^2+dy^2)$ is non-ergodic DBG,  we can find a constant $\delta_0$ such that for all $\delta\in (-\delta_0,\delta_0)$, a torus with metric $(g(r)^2+\delta)(dx^2+dy^2)$ is also non-ergodic DBG. This follows from the fact that being a non-ergodic DBG torus is an open condition.
\end{remark}
\begin{remark}
By choosing a sufficiently large $a$ we can shrink the support of $g'$ to be as small as we want. Indeed notice that $\rho(s)\leq s$ and $\rho'(s)>0$ on $(\frac{1}{\sqrt{5a}}, \frac{1}{2\sqrt{a}})$. Therefore
\begin{eqnarray*}
\max_{s\geq 0} \,\,\,s-\rho(s)&=&\max_{0\leq s\leq \frac{1}{2\sqrt{a}}}\int_0^s1-\rho'(t)dt\\
&=& \frac{1}{\sqrt{5a}}-\rho(\frac{1}{\sqrt{5a}})+\max_{0\leq s\leq \frac{1}{2\sqrt{a}}}\int_{\frac{1}{\sqrt{5a}}}^s 1-\rho'(t)dt\\
&<&  \frac{1}{\sqrt{5a}}-\rho(\frac{1}{\sqrt{5a}})+\max_{0\leq s\leq \frac{1}{2\sqrt{a}}}s-\frac{1}{\sqrt{5a}}<\frac{1}{3\sqrt{a}}.
\end{eqnarray*}

From $(\ast)$ we have 
\begin{eqnarray*}
\ln r_2-\ln (\frac{1}{2\sqrt{a}})&=&\int_0^{\frac{1}{2\sqrt{a}}}\left(\frac{1}{\rho(s)}-\frac{1}{s}\right)ds\\
&=&\int_0^{\frac{1}{\sqrt{5a}}}\left(\frac{1}{s-5as^3+10a^2s^5}-\frac{1}{s}\right)ds+\int_{\frac{1}{\sqrt{5a}}}^{\frac{1}{2\sqrt{a}}}\left(\frac{1}{\rho(s)}-\frac{1}{s}\right)ds
\end{eqnarray*}
\begin{eqnarray*}
&<&\int_0^{\frac{1}{\sqrt{5a}}}\frac{5as-10a^2s^3}{1-5as^2+10a^2s^4}ds+\int_{\frac{1}{\sqrt{5a}}}^{\frac{1}{2\sqrt{a}}} \left(\frac{1}{s-\frac{1}{3\sqrt{a}}}-\frac{1}{s}\right)ds\\
&<& 10a \int_0^{\frac{1}{\sqrt{5a}}}2s-4as^3ds+ \ln\left(1-\frac{1}{3s\sqrt{a}}\right)\Big|_{\frac{1}{\sqrt{5a}}}^{\frac{1}{2\sqrt{a}}}< 2-\ln(3-\sqrt{5}).
\end{eqnarray*}
Thus $r_2\rightarrow 0$ as $a\rightarrow \infty$.
\end{remark}

\section{Perturbation of the Hamiltonian $H_0$}
Suppose the fundamental domain of the deck group on the universal cover of our torus $\mathbb{T}^2\cong \mathbb{R}^2/\mathbb{Z}^2$ is $\{-1< x,y< 1\}$. We use $\alpha, \beta$ to denote the coordinates in the cotangent space and denote $B^{*}\mathbb{T}^2:=\{(x,y,\alpha,\beta)\in T^*\mathbb{T}^2: \alpha^2+\beta^2<1\}$.  In this section we want to perturb the kinetic Hamiltonian
$$H_0(x,y,\alpha,\beta):=\frac{\alpha^2+\beta^2}{2}$$
in such a way that the Hamiltonian flow if the resulting Hamiltonian has positive metric entropy.  More precisely, we want to prove the following:
\begin{lemma}\label{lem3}
There exists a family $\{H_{\epsilon}\}_{\epsilon> 0}$ of smooth perturbations of $H_0$ such that for all $\epsilon>0$, there exists an open interval $I_{\epsilon}$ with the property that for any $h\in I_{\epsilon}$, the Hamiltonian flow $\Phi_{H_{\epsilon}}^t$ on the level set $\{H_{\epsilon}=h\}$ has positive metric entropy. 
\end{lemma}
\begin{proof}
Let $\xi:\mathbb{R}_{\geq 0}\rightarrow [0,1]$ be a smooth function with $\xi\equiv 1$ on $[0,1/3]$ and $\xi\equiv 0$ on $[2/3,1]$. And let $g$ be the function we built in Section 4. We define 
$$H_{\epsilon}:=H_0+\epsilon (1-g(r)^2)\xi(\alpha^2+\beta^2), \text{ where } r=\sqrt{x^2+y^2}. $$
Since $g$ is positive and $0\leq 1-g^2<1$ (by Remark 1), we have 
$$\epsilon>\max_{(x,y)\in \mathbb{T}^2}\epsilon (1-g(r)^2)\xi(\alpha^2+\beta^2).$$
Notice that if $H_{\epsilon}<1/6$ then $\alpha^2+\beta^2<1/3$, therefore $\xi\equiv 1$ whenever the total energy is small. By the Maupertuis principle, the Hamiltonian flow $\Phi^t_{H_{\epsilon}}$ on the level set $\{H_{\epsilon}=\epsilon\}$ is a time change of the geodesic flow on $\mathbb{T}^2$ with metric 
$$ds^2=\epsilon g(r)^2(dx^2+dy^2). $$
This metric has positive metric entropy since, by Lemma \ref{lem1},  the metric $ds^2=g(r)^2(dx^2+dy^2)$ does. 

Let $\delta_0$ be the constant we get from Remark \ref{rem2} and define $I_{\epsilon}:=(\epsilon-\epsilon\delta_0, \epsilon+\epsilon\delta_0)$. By using Maupertuis principle again we prove the lemma.
\end{proof}

\section{Perturbation of $\tilde{H}_0$}
In this section we prove that a smooth perturbation of 
$$\tilde{H}_0(x,y,\alpha,\beta):=-\sqrt{1-\alpha^2-\beta^2}$$
can be derived from a suitable perturbation of $H_0$. Since this result holds for all degrees of freedom, we use $(\textbf{q}, \textbf{p})$ to denote the coordinates instead of $(x,y,\alpha,\beta)$. 
 
Suppose $\mathbb{T}^n=\mathbb{R}^n/\mathbb{Z}^n$ has coordinates $\textbf{q}=(q_1,...,q_n)$ and let $\textbf{p}=(p_1,...,p_n)$ be the coordinates in the cotangent bundle. Denote $B^*\mathbb{T}^n=\{(\textbf{q},\textbf{p}): \sum p^2_i<1\}$. Define
$$H_0(\textbf{q}, \textbf{p}):= \frac{1}{2}\sum_{i=1}^n p^2_i, \,\,\,\,\,\,\,\,\tilde{H}_0(\textbf{q}, \textbf{p}):=-\sqrt{1-2H_0(\textbf{q}, \textbf{p})}.$$
Then
$$\Phi_{\tilde{H}_{0}}^t(\textbf{q}, \textbf{p})=(\textbf{q}+\frac{t\textbf{p}}{\sqrt{1-\sum p^2_i}},\textbf{p}).$$
Let $V(\textbf{q}, \textbf{p})$ be a $C^2$-smooth function on $B^*\mathbb{T}^n$. We perturb $H_0$ and $\tilde{H}_0$ by $V$ in the following way:
$$H_{\epsilon}(\textbf{q},\textbf{p}):=  \frac{1}{2}\sum_{i=1}^n p^2_i+\epsilon V(\textbf{q},\textbf{p}), \,\,\,\,\,\,\,\,\tilde{H}_{\epsilon}(\textbf{q},\textbf{p}):=-\sqrt{1-2H_{\epsilon}(\textbf{q}, \textbf{p})}.$$
Then we have
\begin{lemma}\label{lem5}
If $\text{supp}V\subseteq \{\sum p^2_i\leq C<1\}$ for some $C\in\mathbb{R}_+$, then for every $\delta,m,\mathcal{T}>0$, there exists $\epsilon=\epsilon(V, \delta,m, \mathcal{T})>0$ such that for each $0\leq T \leq \mathcal{T}$ we have
$$||\Phi_{\tilde{H}_{\epsilon}}^T-\Phi_{\tilde{H}_{0}}^T||_{C^m(B^*\mathbb{T}^n)}<\delta.$$
\end{lemma}
\begin{proof}
Denote $\Phi_{\tilde{H}_{\epsilon}}^T(\textbf{q},\textbf{p})-\Phi_{\tilde{H}_{0}}^T(\textbf{q},\textbf{p})$ by $(\Delta \textbf{q}, \Delta \textbf{p})$ as they usually do this in calculus books. Put $(\textbf{q}(t),\textbf{p}(t)):=\Phi_{\tilde{H}_{\epsilon}}^t(\textbf{q},\textbf{p})$. Suppose that ${H}_{\epsilon}(\textbf{q},\textbf{p})=E$. Then
$$\dot {\textbf{q}} (t)=\frac{\partial \tilde{H}_\epsilon}{\partial \textbf{p}}=\frac{\textbf{p}+\epsilon V_{\textbf{p}}}{\sqrt{1-2E}},\,\,\,\,\,\,\,\,\dot {\textbf{p}}(t)=-\frac{\partial \tilde{H}_\epsilon}{\partial \textbf{q}}=-\frac{\epsilon V_{\textbf{q}}}{\sqrt{1-2E}}.$$
If $\sum p^2_i>C$, then $\dot {\textbf{p}}(t)\equiv 0$, hence $\Delta \textbf{p}=0$. Consider the trajectory $(\textbf{q}(t),\textbf{p}(t))$, $V_p$ vanishes along it, hence $\Delta \textbf{q}=0$. Therefore we only need to consider the case $\sum p^2_i\leq C$. Since $V$ is compactly supported we may assume that $\epsilon$ is small enough so that $\sum p^2_i+2\epsilon V<(1+C)/2<1$. In this case
$$\Delta \textbf{p}=\int_0^T \dot{\textbf{p}}(t)dt=-\int_0^T  \frac{\epsilon V_{\textbf{q}}}{\sqrt{1-2E}}dt=-\frac{\epsilon}{\sqrt{1-\sum p^2_i-2\epsilon V}}\int^T_0 V_{\textbf{q}} dt.$$
{\small
\begin{eqnarray*}
&\Delta \textbf{q}& = \int_0^T \dot{\textbf{q}}(t)dt-\frac{\textbf{p}T}{\sqrt{1-\sum p^2_i}}=\int_0^T \left(\dot{\textbf{q}}(0)+\int_0^t \ddot{\textbf{q}}(s)ds \right)dt-\frac{\textbf{p}T}{\sqrt{1-\sum p^2_i}}\\
&=& T\left(\sqrt{1-\sum p^2_i}-\sqrt{1-\sum p^2_i-2\epsilon V}\right)_{\textbf{p}}+\frac{\int_0^T\int_0^t \dot{\textbf{p}}(s)+\epsilon\dot{\textbf{p}}(s)\cdot V_{\textbf{p}\textbf{p}}+ \epsilon\dot{\textbf{q}}(s)\cdot V_{\textbf{q}\textbf{p}}dsdt}{\sqrt{1-2E}}\\
&=& T\left(\sqrt{1-\sum p^2_i}-\sqrt{1-\sum p^2_i-2\epsilon V}\right)_{\textbf{p}}+\frac{\int_0^T\int_0^t -\epsilon V_{\textbf{q}}-\epsilon^2 V_{\textbf{q}}\cdot V_{\textbf{p}\textbf{p}}+\epsilon(\textbf{p}+\epsilon V_{\textbf{p}})\cdot V_{\textbf{q}\textbf{p}}dsdt}{1-\sum p^2_i-2\epsilon V}.
\end{eqnarray*}}
We can see from the above calculation that since $\sum p^2_i+2\epsilon V<(1+C)/2<1$, $(\Delta \textbf{q}, \Delta \textbf{p})$ converges to 0 uniformly in $C^m$ as $\epsilon\rightarrow 0$.
\end{proof}

\section{The Burago-Ivanov Theorem}
Here we use the notions and definitions from \cite{BI}. 

A Finsler metric $\varphi$ on an $n$-dimensional disc $D$ is called \textit{simple} if it satisfies the following three conditions:

(S1) Every pair of points in $D$ is connected by a unique geodesic.

(S2) Geodesics depend smoothly on their endpoints.

(S3) The boundary is strictly convex, that is, geodesics never touch it at their interior points.

Once $(D,\varphi)$ is simple, denote by $U_{in}, U_{out}$ the set of inward, outward pointing unit tangent vectors with base points in $\partial D$ respectively. With any vector $\nu\in U_{in}$, we can associate a unique vector $\beta(\nu)\in U_{out}$, namely the tangent vector of the (unique) geodesic  with initial velocity $\nu$ at its next intersection point with $\partial D$. This defines a map $\beta: U_{in}\rightarrow U_{out}$, which is called the lens map of $\varphi$. If $\varphi$ is reversible, then the lens map is reversible in the following sense: $-\beta(-\beta(\nu))=\nu$ for every $\nu\in U_{in}$.

We denote by $UT^*D$ the unit sphere bundle with respect to the dual norm $\varphi^*$. Let $\mathscr{L}: TD\rightarrow T^*D$ be the Legendre transform of the Lagrangian $\varphi^2/2$. It maps $UTD$ to $UT^*D$. For a tangent vector $\nu\in UT_xD$, its Legendre transform $\mathscr{L}(\nu)$ is the unique covector $\chi\in U_x^*D$ such that $\chi(\nu)=1$.

Then consider subsets $U^*_{in}=\mathscr{L}(U_{in})$ and $U^*_{out}=\mathscr{L}(U_{out})$ of $UT^*D$. The dual lens map of $\varphi$ is the map $\sigma:U^*_{in}\rightarrow U^*_{out}$ given by $\sigma:=\mathscr{L}\circ \beta\circ \mathscr{L}^{-1}$ where $\beta$ is the lens map of $\varphi$. If $\varphi$ is reversible then $\sigma$ is symmetric in the sense that $-\sigma(-\sigma(\chi))=\chi$ for all $\chi\in U^*_{in}$. 

Note that $U^*_{in}$ and $U^*_{out}$ are $(2n-2)$-dimensional submanifolds of $T^*D$. The restriction of the canonical symplectic 2-form of $T^*D$ to $U^*_{in}$ and $U^*_{out}$ determines the symplectic structure. And the dual lens map $\sigma$ is symplectic. In \cite{BI}, Burago and Ivanov proved the following theorem:

\begin{theorem}[Burago-Ivanov \cite{BI}]\label{thm2}
Assume that $n\geq 3$. Let $\varphi$ be a simple metric on $D=D^n$ and $\sigma$ its dual lens map. Let $W$ be the complement of a compact set in $U^*_{in}$. Then every sufficiently small symplectic perturbation $\tilde{\sigma}$ of $\sigma$ such that $\tilde{\sigma}|_W=\sigma|_W$ is realized by the dual lens map of a simple metric $\tilde{\varphi}$ which coincides with $\varphi$ in some neighborhood of $\partial D$.

The choice of $\tilde{\varphi}$ can be made in such a way that $\tilde{\varphi}$ converges to $\varphi$ whenever $\tilde{\sigma}$ converges to $\sigma$ (in $C^{\infty}$). In addition, if $\varphi$ is a reversible Finsler metric and $\tilde{\sigma}$ is symmetric then $\tilde{\varphi}$ can be chosen reversible as well.
\end{theorem}

\section{Perturbation of flat metric}

Let $\varphi_0$ be the Euclidean metric on $\mathbb{T}^{3}$. We regard $\mathbb{T}^3$ as the cube $[-1,1]^3$ with sides identified. Let $T_0:=[-1,1]^2\times\{-1\}$ be the 2-torus on $\mathbb{T}^{3}$ given by the ``bottom face'' of $\mathbb{T}^{3}$, and we use $x,y,\alpha,\beta$ to denote the coordinates in its cotangent bundle. Let $z$ be the vertical coordinate of $\mathbb{T}^{3}$ and $\gamma$  the corresponding coordinate in the cotangent space. 

Define $\Gamma_0:=\{(p,\chi)\in UT^*\mathbb{T}^{3}: p\in T_0, \chi=(\alpha,\beta,\gamma)\in UT^*_p \mathbb{T}^3, \gamma>0 \}$, where $UT^*\mathbb{T}^{3}$ denotes the unit cotangent bundle of $\varphi_0$. $\Gamma_0$ inherits a natural symplectic form from $T^*\mathbb{T}^3$. We set $R_0:\Gamma_0\rightarrow\Gamma_0$ to be the first return map to $\Gamma_0$ of the geodesic flow on $(\mathbb{T}^{3}, \varphi_0)$. Observe that $R_0$ is a symplectomorphism.

\begin{lemma}\label{lem4}
$R_0$ is symplectomorphic to the time-one map of a Hamiltonian flow $\Phi_{\tilde{H}_0}^t$ on $B^*T_0=\{(p,\chi)\in T^*T_0:  \varphi_0(p,\chi)<1\}$. Here the symplectic form on $B^*T_0$ is the restriction of the natural symplectic form on $T^*T_0$.
\end{lemma}
\begin{proof}
For the covectors in $\Gamma_0$, we have $z=0$ and $\gamma=\sqrt{1-\alpha^2-\beta^2}$ since $v$ is a unit covector. $\Gamma_0$ is bijective to $B^*T_0$ via the canonical projection $\Pi: \Gamma_0 \rightarrow B^*T_0$:
$$\Pi(x,y,-1,\alpha,\beta,\gamma)=(x,y,\alpha,\beta).$$

Observe that $\Pi$ is a symplectic bijection between $\Gamma_0$ and $B^*T_0$. Let 
$$R_1:=\Pi \circ F_0\circ \Pi^{-1}: B^*T_0\rightarrow B^*T_0.$$
By a simple calculation we know that the map $R:B^*\mathbb{R}^3\rightarrow B^*\mathbb{R}^3$ defined by
$$R(x,y,\alpha,\beta):=\left(x+\frac{\alpha}{\sqrt{1-\alpha^2-\beta^2}}, y+\frac{\beta}{\sqrt{1-\alpha^2-\beta^2}}, \alpha,\beta\right)$$
is a lift of $R_1$ to the universal cover. Define a function $\tilde{H}_0$ on $B^*T_0$ by
$$\tilde{H}_0=-\sqrt{1-\alpha^2-\beta^2}.$$
It is not hard to see that $R_1=\Phi_{\tilde{H}_0}^1$. 
\end{proof}

By Lemma \ref{lem4}, in order to perturb $R_0$ to get positive metric entropy, we need only to perturb $R_1$. Note that $\Phi_{\tilde{H}_0}^t$ and $\Phi^t_{H_0}$ are the same up to time reparametrization. Let $H_{\epsilon}$ be the perturbation of $H_0$ as in Lemma \ref{lem3}, and define
$$\tilde{H}_{\epsilon}:=-\sqrt{1-2H_{\epsilon}}.$$ 
$\Phi_{\tilde{H}_{\epsilon}}^t$ has the same trajectories as $\Phi_{H_{\epsilon}}^t$, hence $\Phi_{\tilde{H}_{\epsilon}}^t$ has positive metric entropy since $\Phi_{H_{\epsilon}}^t$ does. Since the support of perturbation is contained in $\{\alpha^2+\beta^2<2/3\}$, $\tilde{H}_{\epsilon}\rightarrow\tilde{H}$ in $C^{\infty}$. From Lemma \ref{lem5} we know that  $\Phi_{\tilde{H}_{\epsilon}}^1\rightarrow \Phi_{\tilde{H}}^1=R_1$ in $C^{\infty}$.

\begin{proof}[Proof of Main Theorem]
We want to build a reversible Finsler metric $\varphi_{\epsilon}$ on $\mathbb{T}^3$ such that the first return map to $\Gamma_0$ is $R_1^{\epsilon}:=\Phi_{\tilde{H}_{\epsilon}}^1$. Let $D^{3}$ be the $3$-dimensional ball inscribed in the $3$-dimensional cube $[-1,1]^{3}$ and $\sigma_0: U^*_{in}\rightarrow U^*_{out}$ be the dual lens map of the Euclidean disc $(D^3, \varphi_0)$. Denote with $\Gamma_{\pm}:=\{(x,y,z,\alpha,\beta,\gamma)\in UT^*(\mathbb{R}^{3}): z=\pm 1, \gamma>0\}$ and define the projections $\Pi_{\pm}: \Gamma_{\pm}\rightarrow B^*\mathbb{R}^{2}$ by
$$\Pi_{\pm}(x,y,\pm 1,\alpha,\beta,\gamma)=(x, y, \alpha,\beta).$$
It is clear that both $\Pi_{\pm}$ are symplectic bijections.

Let $K$ be the support of $R_1^{\epsilon}- R_1$. We define a function $\phi_1: \Pi_{-}^{-1}(K)\rightarrow U^*_{in}$ as follows: for $\zeta\in \Pi_{-}^{-1}(K)$,  consider its orbit under the Euclidean geodesic flow. By choosing large $a$ in the construction of $H_{\epsilon}$ if necessary, this orbit will transverse $U^*_{in}$, and we choose $\phi_1(\zeta)$ to be the first intersection, see Figure 4. Analogously we can define the map $\varphi_2: \Pi^{-1}_{+}(R(K))\rightarrow U^*_{out}$ by associating any $\eta\in \Pi^{-1}_{+}(R(K))$ with the first intersection of its backward orbit with $U^*_{out}$. It is clear that both $\phi_1$ and $\phi_2$ are symplectic injections.

\begin{figure}
\begin{tikzpicture}[line cap=round,line join=round,x=1.0cm,y=1.0cm]
\clip(-0.3,-0.3) rectangle (6.3,4.3);
\draw (0.,3.)-- (0.,0.);
\draw (0.,0.)-- (3.,0.);
\draw (3.,0.)-- (3.,3.);
\draw (3.,3.)-- (0.,3.);
\draw (3.,0.)-- (5.,1.);
\draw (5.,1.)-- (5.,4.);
\draw (5.,4.)-- (3.,3.);
\draw (5.,4.)-- (2.,4.);
\draw (2.,4.)-- (0.,3.);
\draw [dash pattern=on 1pt off 1pt] (0.,0.)-- (2.,1.);
\draw [dash pattern=on 1pt off 1pt] (2.,4.)-- (2.,1.);
\draw [dash pattern=on 1pt off 1pt] (2.,1.)-- (5.,1.);
\draw (3.64,1.1) node[anchor=north west] {$T_0$};
\draw (5.16,2.78) node[anchor=north west] {$\mathbb{T}^{3}$};
\draw(2.5,2.) circle (1.5cm);
\draw [shift={(2.5,3.5)}] plot[domain=3.9269908169872414:5.497787143782138,variable=\t]({1.*2.121320343559643*cos(\t r)+0.*2.121320343559643*sin(\t r)},{0.*2.121320343559643*cos(\t r)+1.*2.121320343559643*sin(\t r)});
\draw [shift={(2.5,0.5)},dash pattern=on 1pt off 1pt]  plot[domain=0.7853981633974483:2.356194490192345,variable=\t]({1.*2.121320343559643*cos(\t r)+0.*2.121320343559643*sin(\t r)},{0.*2.121320343559643*cos(\t r)+1.*2.121320343559643*sin(\t r)});
\draw [->] (1.06,0.26) -- (1.26,0.9);
\draw [->] (1.386334519572954,1.304270462633452) -- (1.586334519572954,1.944270462633452);
\draw [dash pattern=on 1pt off 1pt] (1.2411387900355872,0.8396441281138789)-- (1.386334519572954,1.304270462633452);
\begin{scriptsize}
\draw [fill=black] (1.06,0.26) circle (1.5pt);
\draw[color=black] (1.,0.92) node {$\zeta$};
\draw [fill=black] (1.386334519572954,1.304270462633452) circle (1.5pt);
\draw[color=black] (1.92,2.08) node {$\phi_1(\zeta)$};
\end{scriptsize}
\end{tikzpicture}
\caption{}
\end{figure}

The restriction of $R$ on $K$ can be decomposed as
$$R|_{K}=\Pi_{+}\circ \phi_2^{-1} \circ \sigma_0 \circ \phi_1\circ\Pi_{-}^{-1}.$$

Let $R^{\epsilon}$ be a lift of $R_1^{\epsilon}$ to the universal cover. Define a dual lens map $\sigma_{\epsilon}: U^*_{in}\rightarrow U^*_{out}$ by 
$$\sigma_{\epsilon}(\chi)=\left\{
\begin{aligned}
&\phi_2\circ \Pi_{+}^{-1}\circ R^{\epsilon}\circ \Pi_{-}\circ \phi_1^{-1}(\chi),& &\text{ if } \chi\in\phi_1(\Pi_{-}^{-1}(K)); \\
&-\phi_1\circ \Pi_{-}^{-1}\circ (R^{\epsilon})^{-1} \circ \Pi_{+}\circ \phi_2^{-1}(-\chi), & &  \text{ if } \chi\in -\phi_2(\Pi_{+}^{-1}(R(K))); \\
&\sigma(\chi), & &\text{ otherwise.}
\end{aligned}
\right.$$
It is clear that $\sigma_{\epsilon}$ is symmetric and coincides with $\sigma$ outside a compact set. Moreover $\sigma_{\epsilon}\rightarrow \sigma$ in $C^{\infty}$ as $R^{\epsilon}\rightarrow R$ in $C^{\infty}$.

The map $\sigma_{\epsilon}$ is a symplectic perturbation of $\sigma_0$ and $\sigma_{\epsilon}=\sigma_0$ outside a compact set in $U^*_{in}$. By Theorem \ref{thm2}, there exists a reversible Finsler metric $\varphi_{\epsilon}$ in $D^{3}$ that agrees with $\varphi_0$ in a neighborhood of the boundary $\partial D^{3}$ and such that the dual lens map for $(D^{3}, \varphi_{\epsilon})$ is $\sigma_{\epsilon}$. Now extend $\varphi_{\epsilon}$ to the whole $T\mathbb{T}^{3}$ by setting it equal to $\varphi_0$ outside $D^3$. It has positive metric entropy since the return map does. As $\epsilon \rightarrow 0$, we have $\varphi_{\epsilon} \rightarrow \varphi_0$ in $C^{\infty}$.
\end{proof}
\begin{remark}\label{rem4}
The example we construct in the main theorem does not have Arnold diffusion. In fact, since $\varphi$ is close to the flat metric, we have only to prove that the return map on $\Gamma_0$ cannot have large range in action variables $\alpha, \beta$. This is clear, since for $\epsilon>0$ small enough and for $\alpha^2+\beta^2\geq 1$ the return map coincides with $\tilde{F}$ hence the ranges of action variables are uniformly bounded.

\end{remark}

\section{Acknowledgments}
The author thanks Dmitri Burago for numerous helpful conversations. In particular he suggested this topic and offered help on this problem. The author is grateful to  Moisey Guysinsky, Sergei Ivanov, Anatole Katok, Mark Levi, Federico Rodriguez Hertz and Yakov Sinai for useful discussions. The author thanks the anonymous referee for the helpful comments in revising this paper.

\end{document}